\documentclass[11pt]{elsarticle}
\usepackage{geometry, enumerate, amsmath, amssymb, amsthm, amscd, mathrsfs, hyperref}
\usepackage{color}
\numberwithin{equation}{section}

\newtheorem{Lem}[equation]{Lemma}
\newtheorem{Prop}[equation]{Proposition}
\newtheorem{Thm}[equation]{Theorem}
\newtheorem{Cor}[equation]{Corollary}
\theoremstyle{definition}\newtheorem{Def}[equation]{Definition}
\newtheorem{Constr}[equation]{Construction}
\newtheorem{Not}[equation]{Notation}
\newtheorem{DefNot}[equation]{Definitions-Notations}
\newtheorem{Ex}[equation]{Example}
\newtheorem{Rem}[equation]{Remark}
\newtheorem{Ques}[equation]{Question}

\newcommand{\Q}{\mathbb{Q}}
\newcommand{\Z}{\mathbb{Z}}
\newcommand{\F}{\mathbb{F}}
\newcommand{\N}{\mathbb{N}}
\newcommand{\Int}{\textnormal{Int}}
\newcommand{\bfi}{\mathbf{i}}
\newcommand{\bfj}{\mathbf{j}}
\newcommand{\bfk}{\mathbf{k}}

\newcommand{\mcP}{\mathcal{P}}
\newcommand{\msP}{\mathscr{P}}

\journal{Journal of Algebra}

\makeatletter
\def\ps@pprintTitle{%
     \let\@oddhead\@empty
     \let\@evenhead\@empty
     \def\@oddfoot{\reset@font\hfil\thepage\hfil%
     \llap{\footnotesize\itshape\today}}
     \let\@evenfoot\@oddfoot}
\makeatother

\begin{document}

\begin{frontmatter}

\title{Properly Integral Polynomials over the Ring of Integer-valued Polynomials on a Matrix Ring\\
\vskip0.3cm
\small{to appear in J. Algebra (2016)} 
}

\author{Giulio Peruginelli}
\ead{gperugin@math.unipd.it}
\address{Department of Mathematics, University of Padova, Via Trieste, 63
35121 Padova, Italy}

\author{Nicholas J. Werner}
\ead{wernern@oldwestbury.edu}
\address{Department of Mathematics, Computer and Information Science, SUNY College at Old Westbury, Old Westbury, NY 11568}

\begin{abstract}
\noindent Let $D$ be a domain with fraction field $K$, and let $M_n(D)$ be the ring of $n \times n$ matrices with entries in $D$. The ring of integer-valued polynomials on the matrix ring $M_n(D)$, denoted $\Int_K(M_n(D))$, consists of those polynomials in $K[x]$ that map matrices in $M_n(D)$ back to $M_n(D)$ under evaluation. It has been known for some time that $\Int_\Q(M_n(\Z))$ is not integrally closed. However, it was only recently that an example of a polynomial in the integral closure of $\Int_\Q(M_n(\Z))$ but not in the ring itself appeared in the literature, and the published example is specific to the case $n=2$. In this paper, we give a construction that produces polynomials that are integral over $\Int_K(M_n(D))$ but are not in the ring itself, where $D$ is a Dedekind domain with finite residue fields and $n \geq 2$ is arbitrary. We also show how our general example is related to $P$-sequences for $\Int_K(M_n(D))$ and its integral closure in the case where $D$ is a discrete valuation ring.\\


\noindent MSC Primary 13F20 Secondary 13B22, 11C99
\end{abstract}
\end{frontmatter}

\section{Introduction}
When $D$ is a domain with field of fractions $K$, the ring of integer-valued polynomials on $D$ is $\Int(D) = \{f \in K[x] \mid f(D) \subseteq D\}$. Such rings have been extensively studied over the past several decades; the reader is referred to \cite{CahCha} for standard results on these objects. More recently, attention has turned to the consideration of integer-valued polynomials on algebras \cite{ChabPer, EvrFarJoh, EvrJoh, Fri, FriCorr, Fri3, LopWer, PerDivDiff, Per, PerIntvalalg, PerWer, Wer3}. The typical approach for this construction is to take a torsion-free $D$-algebra $A$ that is finitely generated as a $D$-module and such that $A \cap K = D$. Then, we define $\Int_K(A)$ to be the set of polynomials in $K[x]$ that map elements of $A$ back to $A$ under evaluation. That is, $\Int_K(A) := \{f \in K[x] \mid f(A) \subseteq A\}$, which is a subring of $\Int(D)$. (Technically, evaluation of $f \in K[x]$ at elements of $A$ is performed in the tensor product $K \otimes_D A$ by associating $K$ and $A$ with their canonical images $K \otimes 1$ and $1 \otimes A$. In practice, however, it is usually clear how to perform the evaluation without the formality of tensor products.)

Depending on the choice of $A$, the ring $\Int_K(A)$ can exhibit similarities to, or stark differences from, $\Int(D)$. For instance, if $A$ is the ring of integers of a number field (viewed as a $\Z$-algebra), then $\Int_\Q(A)$ is---like $\Int(\Z)$---a Pr\"{u}fer domain \cite[Thm.\ 3.7]{LopWer}, hence is integrally closed. In contrast, when $A = M_n(\Z)$ is the algebra of $n \times n$ matrices with entries in $\Z$, $\Int_\Q(A)$ is not integrally closed (although its integral closure is a Pr\"{u}fer domain) \cite[Sec.\ 4]{LopWer}. In a more general setting, it is known \cite[Thm.\ VI.1.7]{CahCha} that if $D$ is a Dedekind domain with finite residue fields, then $\Int(D)$ is a Pr\"{u}fer domain, and so is integrally closed. The motivation for this paper was to show, by giving a form for a general counterexample, that $\Int_K(M_n(D))$ is not integrally closed. In this vein, we make the following definition.

\begin{Def}\label{Properly integral}
A polynomial $f \in K[x]$ will be called \textit{properly integral} over $\Int_K(A)$ if $f$ lies in the integral closure of $\Int_K(A)$, but $f \notin \Int_K(A)$.
\end{Def}

Note that the integral closure of $\Int_K(A)$ in its field of fractions $K(x)$ is contained in $K[x]$, so that $\Int_K(A)$ is integrally closed if and only if there are no properly integral polynomials over $\Int_K(A)$. It has been known for some time that $\Int_\Q(M_n(\Z))$ is not integrally closed. However, the first published example of a properly integral polynomial over $\Int_\Q(M_n(\Z))$ was given only recently by Evrard and Johnson in \cite{EvrJoh}, and only for the case $n = 2$. We will give a general construction for a properly integral polynomial over $\Int_K(M_n(D))$, where $D$ is a Dedekind domain with finite residue rings, and $n \geq 2$ is arbitrary.

The theorems in this paper can be seen as complementary to the work of Evrard and Johnson. Their results relied heavily on the $P$-orderings and $P$-sequences of Bhargava \cite{Bha} and the generalizations of these in \cite{Joh}. In the case where $D = \Z$, a properly integral polynomial $f(x) = g(x)/p^k$ (where $g \in \Z[x]$, $p^k$ is a prime power, and $p$ does not divide $g$) over $\Int_\Q(M_n(\Z))$ produced by using the methods and $p$-sequences in \cite{EvrJoh} is optimal in the sense that $f$ has minimal degree among all properly integral polynomials of the form  $g_1(x)/p^{k_1}$, where $k_1>0$.  
However, building such an $f$ requires knowing the $p$-sequences for $\Int_\Q(M_n(\Z))$ and its integral closure. In general, these sequences are  quite difficult to determine; to date, formulas for such $p$-sequences have been given only in the case $n=2$. In contrast, our construction gives a properly integral polynomial for a much larger variety of rings and does not require a $P$-sequence, but it is only known to be optimal when $n=2$ and $D=\Z$---a fact we can prove precisely because of the $p$-sequences derived in \cite{EvrJoh}.

The paper proceeds as follows. Section \ref{Construction section} begins with a concrete construction for a properly integral polynomial over $\Int_K(M_n(D))$ that works when $D$ is a discrete valuation ring (DVR). This local result is then globalized (Theorem \ref{Dedekind case}) in Section \ref{Global section} to the case where $D$ is a Dedekind domain. We also point out (Corollary \ref{Algebra case}) that the same construction works for some algebras that are not matrix rings. Section \ref{Null ideals} relates our work to the $P$-sequences used by Evrard and Johnson.  We generalize (Theorem \ref{Phi thm}) a classical theorem known to Dickson \cite[Thm.\ 27, p.\ 22]{Dickson} concerning the ideal of polynomials in $\Z[x]$ whose values over $\Z$ are divisible by a fixed prime power $p^k$, $k\leq p$, and use this generalization to give a concise formula (Corollary \ref{Formula for mu_d}) for the initial terms of the $P$-sequence for $\Int_K(M_n(V))$, where $V$ is a DVR and $n \geq 2$. Finally, by utilizing the formulas given in \cite{EvrJoh} for the $p$-sequences of the integral closure of $\Int_\Q(M_2(\Z))$, we prove that the polynomials produced by our construction are optimal (in the sense of the previous paragraph) when $D = \Z$ and $n = 2$ (Corollary \ref{When F is optimal}).

\section{Construction of the Properly Integral Polynomial}\label{Construction section}

Let $V$ be a discrete valuation ring (DVR) with maximal ideal $\pi V$, field of fractions $K$, and finite residue field $V/\pi V \cong \F_q$. Fix an algebraic closure $\overline{K}$ of $K$ and for each $n \geq 2$, let $\Lambda_n(V)$ be the set of elements of $\overline{K}$ whose degree over $V$ is at most $n$. For each $\alpha \in \Lambda_n(V)$, we let $O_\alpha$ be the integral closure of $V$ in $K(\alpha)$. 

We know \cite[Cor.\ 16]{PerWer} that the integral closure of $\Int_K(M_n(V))$ is equal to
\begin{equation*}
\Int_K(\Lambda_n(V)) := \{ f \in K[x] \mid f(\Lambda_n(V)) \subseteq \Lambda_n(V)\}.
\end{equation*}
Note that since $\Lambda_n(V)\cap K(\alpha)=O_{\alpha}$, we have $f \in \Int_K(\Lambda_n(V))$ if and only if $f(\alpha) \in O_\alpha$ for each $\alpha \in \Lambda_n(V)$. Here, we will give a general construction for a polynomial $F$ that is properly integral over $\Int_K(M_n(V))$; that is, $F \in \Int_K(\Lambda_n(V)) \setminus \Int_K(M_n(V))$. The idea behind the construction is as follows.

A polynomial $f \in K[x]$ is integer-valued on $M_n(V)$ if and only if it is integer-valued on the set of $n \times n$ companion matrices in $M_n(V)$ \cite[Thm.\ 6.3]{Fri}. It turns out that if $f$ is integer-valued on ``enough'' companion matrices, then it can still lie in $\Int_K(\Lambda_n(V))$. However, as long as $f$ is not integer-valued on at least one companion matrix, $f$ will not be in $\Int_K(M_n(V))$. So, we will build a polynomial that is integer-valued on almost all of the companion matrices in $M_n(V)$; specifically, our polynomial will fail to be integer-valued on the set of companion matrices whose characteristic polynomial mod $\pi$ is a power of a linear polynomial.

As part of our construction, we will lift elements from $\F_q$ or $\F_q[x]$ up to $V$ or $V[x]$. To be precise, one should first pick residue representatives for $\F_q$ and $\F_q[x]$, and then use these in all calculations taking place over $V$. However, to ease the notation, we will write $\F_q$ throughout. When a calculation must be performed over the finite field, we will say it occurs ``mod $\pi$'' or ``in $\F_q$''.

\begin{Constr}\label{The construction}
Let $n \geq 2$. Let 
\begin{align*}
\msP &= \{f \in \F_q[x] \mid f \text{ is monic, irreducible, and } 2 \leq \deg f \leq n\},\\
\theta(x) &= \prod_{f \in \msP} f(x)^{\lfloor n/\deg f \rfloor},\\ 
h(x) &= x^{n-1}\prod_{a \in \F_q^\times}(x^n + \pi a),\\
H(x) &= \prod_{b \in \F_q} h(x-b), \text{ and finally}\\
F(x) &= \dfrac{H(x)(\theta(x))^q}{\pi^q}.
\end{align*}
\end{Constr}

The simplest example for $F$ occurs when $n=2$ and $V = \Z_{(2)}$, so that $q = 2$. Then, we have
\begin{equation*}
F(x) = \dfrac{x(x^2+2)(x-1)((x-1)^2+2)(x^2+x+1)^2}{4}.
\end{equation*}

Our main result is the following.

\begin{Thm}\label{Big thm}
Let $F$ be as in Construction \ref{The construction}. Then, $F$ is properly integral over $\Int_K(M_n(V))$.
\end{Thm}

\begin{Rem}\label{Uniqueness remark}
Even if one specifies a degree $d_0$ and a denominator $d$, properly integral polynomials of the form $g(x)/d$ with $\deg g = d_0$ are not unique. Indeed, \cite[Cor.\ 3.6]{EvrJoh} shows that
\begin{equation*}
G(x) = \dfrac{x(x^2+2x+2)(x-1)(x^2+1)(x^2-x+1)(x^2+x+1)}{4}
\end{equation*}
is properly integral over $\Int_\Q(M_2(\Z_{(2)}))$, and clearly $G$ is not equal to the $F$ given above. However, one may prove the following. Let $I$ be the ideal of $\Z_{(2)}[x]$ generated by $4$ and $2x^2(x-1)^2(x^2+x+1)$. If $G_1, G_2 \in \Z_{(2)}[x]$ are both monic of degree 10 and both $F_1 = G_1/4$ and $F_2 = G_2/4$ are properly integral over $\Int_\Q(M_2(\Z_{(2)}))$, then $G_1$ and $G_2$ are equivalent modulo $I$. As one may check, this is the case with $F$ and $G$. Similar equivalences are possible for other choices of $V$, $d_0$, and $d$.
\end{Rem}

One part of Theorem \ref{Big thm} is easy to prove. The remaining parts are more involved, and the proof is completed at the end of Section \ref{Construction section}.

\begin{Lem}\label{Big thm first lemma}
$F\notin \Int_K(M_n(V))$.
\end{Lem}
\begin{proof}
Let $C \in M_n(V)$ be the companion matrix for $x^n$. Then, $(\prod_{a \in \F_q^\times}h(C-aI)) (\theta(C))^q$ is a unit mod $\pi$, hence is also a unit mod $\pi^q$. So, the only way that $F(C)$ will be in $M_n(V)$ is if $h(C)$ is 0 mod $\pi^q$. However,
\begin{equation*}
h(C) = C^{n-1}\prod_{a \in \F_q^\times}(C^n + \pi aI) = C^{n-1}\prod_{a \in \F_q^\times}(\pi a)I
\end{equation*}
is only divisible by $q-1$ powers of $\pi$. Thus, $h(C) \notin \pi^q M_n(V)$, $F(C) \notin M_n(V)$, and $F \notin \Int_K(M_n(V))$.
Note that the same steps work if we replace $x^n$ by $(x-a)^n$, where $a \in \F_q^\times$, and replace $h(x)$ with $h(x-a)$.
\end{proof}

\begin{Rem}\label{Pullback remark}
The failure of $F$ to lie in $\Int_K(M_n(V))$ can also be expressed in terms of pullback rings. By \cite[Rem.\ 2.1 \& (3)]{PerDivDiff}), we have
\begin{equation*}
\Int_K(M_n(V))=\bigcap_{f \in \mcP_n}(V[x]+f(x)K[x]),
\end{equation*}
where $\mcP_n$ is the set of monic polynomials in $V[x]$ of degree exactly equal to $n$. The previous Claim then demonstrates that $F \notin V[x] + (x-a)^n K[x]$ for any $a \in \F_q$.
\end{Rem}

Showing that $F \in \Int_K(\Lambda_n(V))$ is more difficult. The general idea is to take $\alpha \in \Lambda_n(V)$ and focus on its minimal polynomial $m(x)$ over $V$. We then consider two possibilities, according to how the polynomial $m(x)$ factors over the residue field. Either $m(x) \equiv (x-a)^n$ mod $\pi$, for some $a \in \F_q$; or $m(x) \not\equiv (x-a)^n$ mod $\pi$, for all $a \in \F_q$. The first case is the more difficult one, and occupies the next several results. The second case is dealt with in Lemma \ref{Big thm second lemma}, right before we complete the proof of Theorem \ref{Big thm}.

So, for now we concentrate on those cases where $m(x) \equiv (x-a)^n$ mod $\pi$ for some $a \in \F_q$. In fact, by translation, it will be enough to consider the case where $m(x) \equiv x^n$ mod $\pi$. Our starting point is a lemma involving symmetric polynomials. Given a set $S = \{x_1, x_2, \ldots, x_n\}$, for each $1 \leq k \leq n$, we let $\sigma_k(S)$ denote the $k^{\text{th}}$ elementary symmetric polynomial in $x_1, x_2, \ldots, x_n$. 

\begin{Lem}\label{Symm poly lem}
Let $\overline{V}$ be the integral closure of $V$ in $\overline{K}$. Let $n \geq 2$, let $S = \{\alpha_1, \alpha_2, \ldots, \alpha_n\} \subset\overline{V}$, and let $S^{n-1} = \{\alpha_1^{n-1}, \alpha_2^{n-1}, \ldots, \alpha_n^{n-1}\}$. Assume the following:
\begin{itemize}
\item $\sigma_k(S) \in \pi V$, for all $1 \leq k \leq n-1$.
\item $\sigma_n(S) \in \pi^2 V$.
\end{itemize}
Then, $\sigma_k(S^{n-1}) \in \pi^k V$ for each $1 \leq k \leq n$.
\end{Lem}
\begin{proof}
The stated conditions ensure that $\prod_{k=1}^n(x-\alpha_k) \in V[x]$. In particular, $S \subset \Lambda_n(V)$ and the set of conjugates of each $\alpha_k $ is contained in $S$. From this, it follows that the polynomial $\prod_{k=1}^n(x-\alpha_k^{n-1})$ is also in $V[x]$, and thus that $\sigma_k(S^{n-1}) \in V$ for each $k$. It remains to show that $\pi^k$ divides $\sigma_k(S^{n-1})$.

Fix $k$ between 1 and $n$ and consider $\sigma_k(S^{n-1})$. By ``total degree'' we mean degree as a polynomial in $\alpha_1, \alpha_2, \ldots, \alpha_n$. Thus, each element of $S^{n-1}$ has total degree $n-1$; each monomial of $\sigma_k(S^{n-1})$ has total degree $k(n-1)$; and $\sigma_k(S^{n-1})$, being homogeneous in $\alpha_1, \alpha_2, \ldots, \alpha_n$, also has total degree $k(n-1)$. By the Fundamental Theorem of Symmetric Polynomials, $\sigma_k(S^{n-1})$ equals a polynomial $f$ in $\sigma_1(S), \sigma_2(S), \ldots, \sigma_n(S)$. Moreover, each monomial $a\sigma_1(S)^{e_1}\sigma_2(S)^{e_2} \cdots \sigma_n(S)^{e_n}$ in $f$ has total degree (as a polynomial in $\alpha_1, \alpha_2, \ldots, \alpha_n$) equal to $k(n-1)$. It suffices to prove that each such monomial in $f$ is divisible by $\pi^k$.

Let $v$ denote the natural valuation for $V$ and let $\beta = a\sigma_1(S)^{e_1}\sigma_2(S)^{e_2} \cdots \sigma_n(S)^{e_n}$ be a monomial of $f$ in $\sigma_1(S), \sigma_2(S), \ldots, \sigma_n(S)$. Since the total degree of $\beta$ is $k(n-1)$, we obtain 
\begin{equation}\label{eq1}
e_1 + 2e_2 + \cdots + n e_n = k(n-1).
\end{equation}
Also, by assumption,
\begin{align*}
v(\beta) &\geq v(\sigma_1(S)^{e_1}) + v(\sigma_2(S)^{e_2}) + \cdots + v(\sigma_n(S)^{e_n})\\
&\geq e_1 + e_2 + \cdots + e_{n-1} + 2e_n.
\end{align*}
We want to show that $v(\beta) \geq k$. From \eqref{eq1}, we have
\begin{align*}
k &= \dfrac{e_1}{n-1} + \dfrac{2 e_2}{n-1} + \cdots + \dfrac{(n-1) e_{n-1}}{n-1} + \dfrac{n e_n}{n-1}\\
&\leq e_1 + e_2 + \cdots + e_{n-1} + 2 e_n\\
&\leq v(\beta),
\end{align*}
as desired.
\end{proof}

Lemma \ref{Symm poly lem} is used to prove the first part of the next proposition.

\begin{Prop}\label{Alg int prop}
Let $\alpha \in \Lambda_n(V)$ with minimal polynomial over $V$ equal to $m(x) = x^n + \pi a_{n-1}x^{n-1} + \cdots + \pi a_1 x + \pi a_0$, where each $a_k \in V$.
\begin{enumerate}[(1)]
\item If $a_0 \in \pi V$, then $\alpha^{n-1}/\pi \in O_\alpha$.
\item If $a_0 \notin \pi V$, then $(\alpha^{n-1}(\alpha^n + \pi a))/\pi^2 \in O_\alpha$, where $a \in \F_q$ is the residue of $a_0$ mod $\pi$.
\end{enumerate}
\end{Prop}
\begin{proof}
(1) Assume $a_0 \in \pi V$. Let $\alpha_1, \alpha_2, \ldots, \alpha_n$ be the roots of $m$. Let $S = \{\alpha_1, \alpha_2, \ldots, \alpha_n\}$ and $S^{n-1} = \{\alpha_1^{n-1}, \alpha_2^{n-1}, \ldots, \alpha_n^{n-1}\}$. Then, the conditions of Lemma \ref{Symm poly lem} are satisfied, so $\sigma_k(S^{n-1}) \in \pi^k V$ for each $1 \leq k \leq n$. Let $g(x) = \prod_{k=1}^n (x-\alpha_k^{n-1}/\pi)$. Then, the coefficient of $x^{n-k}$ in $g(x)$ is $(-1)^{n-k} \sigma_k(S^{n-1})/\pi^k \in V$. Thus, $g \in V[x]$ and $g(\alpha^{n-1}/\pi) = 0$, so $\alpha^{n-1}/\pi \in O_\alpha$.\\

(2) Assume $a_0 \notin \pi V$. Let $a' \in V$ be such that $a - a_0 = \pi a'$. Since $m(\alpha) = 0$, we have
\begin{equation*}
\alpha^n = -\pi a_{n-1}\alpha^{n-1} - \cdots - \pi a_1\alpha - \pi a_0.
\end{equation*}
In particular, this means that $\alpha^n \in \pi O_\alpha$. Next,
\begin{align*}
\alpha^{n-1}(\alpha^n + \pi a) &= \alpha^{n-1}(-\pi a_{n-1}\alpha^{n-1} - \cdots - \pi a_1\alpha - \pi a_0 + \pi a)\\
&= \alpha^{n-1}(-\pi a_{n-1}\alpha^{n-1} - \cdots - \pi a_1\alpha + \pi^2 a').
\end{align*}
For each $1 \leq k \leq n-1$, $\alpha^{n-1}(\pi a_k \alpha^k)$ is divisible by both $\pi$ and $\alpha^n$, so $\alpha^{n-1}(\pi a_k \alpha^k) \in \pi^2 O_\alpha$. Also, $\alpha^{n-1}\pi^2 a' \in \pi^2 O_\alpha$. It now follows that $(\alpha^{n-1}(\alpha^n + \pi a))/\pi^2 \in O_\alpha$.
\end{proof}

Now, we relate Proposition \ref{Alg int prop} to the polynomial from Construction \ref{The construction}.

\begin{Prop}\label{Alg int prop 2}
Let $\alpha \in \Lambda_n(V)$ have minimal polynomial $m(x)$ such that $m(x) \equiv x^n$ mod $\pi$. Let $f(x) = h(x)/\pi^q$, where $h$ is as in Construction \ref{The construction}. Then, $f(\alpha) \in O_\alpha$.
\end{Prop}
\begin{proof}
Since $m(x) \equiv x^n$ mod $\pi$, we have $m(x) = x^n + \pi a_{n-1}x^{n-1} + \cdots + \pi a_1 x + \pi a_0$ for some $a_0, \ldots, a_{n-1} \in V$. Note that for $a \in \F_q^\times$, we have $(\alpha^n + \pi a)/\pi \in O_\alpha$.

If $a_0 \in \pi V$, then $\alpha^{n-1}/\pi \in O_\alpha$ by Proposition \ref{Alg int prop} part (1). In this case,
\begin{equation*}
f(\alpha) = \dfrac{\alpha^{n-1}}{\pi}\prod_{a \in \F_q^\times} \dfrac{\alpha^n + \pi a}{\pi}
\end{equation*}
is an element of $O_\alpha$.

If $a_0 \notin \pi V$, then by Proposition \ref{Alg int prop} part (2), there exists $a \in \F_q^\times$ such that $(\alpha^{n-1}(\alpha^n + \pi a))/\pi^2 \in O_\alpha$. This time, we group the factors of $f(\alpha)$ as
\begin{equation*}
f(\alpha) = \dfrac{\alpha^{n-1}(\alpha^n + \pi a)}{\pi^2}  \prod_{\substack{b \in \F_q^\times,\\b\ne a}} \dfrac{\alpha^n + \pi b}{\pi}
\end{equation*}
and as before we see that $f(\alpha) \in O_\alpha$.
\end{proof}

Proposition \ref{Alg int prop 2} is what we ultimately need to prove Theorem \ref{Big thm}. As mentioned after Remark \ref{Pullback remark}, there is a second case to consider, in which an element $\alpha \in \Lambda_n(V)$ has a minimal polynomial $m(x)$ such that $m(x) \not\equiv (x-a)^n$ mod $\pi$, for all $a \in \F_q$. Most of the work required in this case is done in the next lemma.

\begin{Lem}\label{Big thm second lemma}
Let $\alpha \in \Lambda_n(V)$ with minimal polynomial $m(x)$ such that $m(x) \not\equiv (x-a)^n$ mod $\pi$, for all $a \in \F_q$. Then  the numerator of $F$ admits a factorization $H(x)(\theta(x))^q = \prod_{b \in \F_q} f_b(x)$ such that $m(x)$ divides $f_b(x)$ mod $\pi$ for each $b$. Consequently, $F(\alpha)\in O_{\alpha}$.
\end{Lem}
\begin{proof}
We have
\begin{align*}
H(x) &= \prod_{b \in \F_q} h(x-b)\\
&= \prod_{b \in \F_q} \Big[(x-b)^{n-1} \prod_{a \in \F_q^\times} ((x-b)^n + \pi a)\Big] \\
&= \Big[ \prod_{b \in \F_q} (x-b)^{n-1}\Big]\Big[\prod_{a \in \F_q^\times} \prod_{b \in \F_q} ((x-b)^n + \pi a)\Big].
\end{align*}
Let $f_0(x) = (\prod_{b \in \F_q} (x-b)^{n-1})\theta(x)$ and for each $a \in \F_q^\times$, let $f_a(x) = (\prod_{b \in \F_q} ((x-b)^n + \pi a))\theta(x)$. Then, $H(x)(\theta(x))^q = \prod_{b \in \F_q} f_b(x)$.

Now, factor $m(x)$ mod $\pi$ as
\begin{equation*}
m(x) \equiv \iota_1(x)^{n_1} \iota_2(x)^{n_2} \cdots \iota_t(x)^{n_t}
\end{equation*}
where each $\iota_k$ is a distinct monic irreducible polynomial in $\F_q[x]$. Assuming that $m(x) \not\equiv (x-a)^n$ mod $\pi$ for any $a \in \F_q$, each exponent $n_k$ satisfies $1 \leq n_k \leq \lfloor n/\deg(\iota_k) \rfloor  < n$ if  $\deg(\iota_k)>1$ or $1 \leq n_k <n$ if $\deg(\iota_k)=1$. So, working mod $\pi$, the product of the $\iota_k^{n_k}$ with $\deg(\iota_k) > 1$ divides $\theta$; and the product of the $\iota_k^{n_k}$ with $\deg(\iota_k) = 1$ divides $f_b/\theta$ for each $b \in \F_q$. Hence, $m(x)$ divides $f_b(x)$ mod $\pi$. Finally, the last condition implies that $f_b(\alpha) \in \pi O_\alpha$ for each $b$. Thus, $H(\alpha)(\theta(\alpha))^q = \prod_{b \in \F_q} f_b(\alpha) \in \pi^q O_\alpha$, and so $F(\alpha) \in O_\alpha$.
\end{proof}

Finally, we complete the proof of Theorem \ref{Big thm}. For convenience, the theorem is restated below.

\setcounter{equation}{1}
\begin{Thm}
Let the notation be as in Construction \ref{The construction}. Then, $F$ is properly integral over $\Int_K(M_n(V))$.
\end{Thm}
\begin{proof}
The polynomial $F \notin \Int_K(M_n(V))$ by Lemma \ref{Big thm first lemma}. To show that $F \in \Int_K(\Lambda_n(V))$, let $\alpha \in \Lambda_n(V)$. We will prove that $F(\alpha) \in O_\alpha$.

Let $m(x)$ be the minimal polynomial of $\alpha$. If $m(x) \equiv (x-a)^n$ mod $\pi$ for some $a \in \F_q$, then by Proposition \ref{Alg int prop 2} we have $h(\alpha - a)/\pi^q \in O_\alpha$. Hence, $F(\alpha) \in O_\alpha$ in this case. If instead $m(x) \not\equiv (x-a)^n$ mod $\pi$ for all $a \in \F_q$, then by Lemma  \ref{Big thm second lemma} we still have  $F(\alpha) \in O_\alpha$.
\end{proof}
\setcounter{equation}{10}

\section{Globalization and Extension to Algebras}\label{Global section}
In this section, we discuss how to globalize Construction \ref{The construction}, and demonstrate that it is applicable to algebras other than matrix rings.

Thus far, we have focused on the local case and worked with the DVR $V$. However, since the formation of our integer-valued polynomial rings is well-behaved with respect to localization, our results can be applied to the global case where $V$ is replaced with a Dedekind domain. For the remainder of this section, $D$ will denote a Dedekind domain with finite residue fields. As with $V$, we let $K$ be the fraction field of $D$ and we fix an algebraic closure $\overline{K}$ of $K$. For $n \geq 2$, let $\Lambda_n(D)$ be the set of elements of $\overline{K}$ whose degree over $D$ is at most $n$. Then, by \cite[Cor.\ 16]{PerWer}, the integral closure of $\Int_K(M_n(D))$ is $\Int_K(\Lambda_n(D))$. By taking $V = D_P$ for a nonzero prime $P$ of $D$, we can use our local construction to produce polynomials that are properly integral over $\Int_K(M_n(D))$. Most of the work is done in the following lemma, which works over any integral domain $D$.

\begin{Lem}\label{Global lemma}
Let $R$ and $S$ be $D$-modules such that $D[x] \subseteq R \subseteq S \subseteq K[x]$. Assume there exists a nonzero prime $P$ of $D$ such that $R_{P} \subsetneqq S_{P}$, and let $f \in S_{P} \setminus R_{P}$. Then, there exists $c \in D \setminus P$ such that $cf \in S \setminus R$.
\end{Lem}
\begin{proof}
Write $f(x) = g(x)/d$, where $g \in S$, $d \in D\setminus P$, and $d$ does not divide $g$.  Since $f(x)$ is not in $R_P$, $g\notin R$. Hence, $df=g\in S\setminus R$, as wanted.
\end{proof}

We also require a result regarding the localization of $\Int_K(A)$ at primes of $D$.

\begin{Prop}\label{Localization prop}(\cite[Prop.\ 3.1, 3.2]{Wer3})
Let $A$ be a torsion-free $D$-algebra that is finitely generated as $D$-module and such that $A \cap K = D$. Then, $\Int_K(A)_Q = \Int_K(A_Q)$ for each nonzero prime $Q$ of $D$, and $\Int_K(A) = \bigcap_Q \Int_K(A)_Q$, where the intersection is over all nonzero primes $Q$ of $D$.
\end{Prop}

Combining Construction \ref{The construction} with Lemma \ref{Global lemma} now allows us to produce polynomials that are properly integral over $\Int_K(M_n(D))$.

\begin{Thm}\label{Dedekind case}
Let $P$ be a nonzero prime of $D$. Let $V = D_P$ and let $F$ be the polynomial from Construction \ref{The construction} applied to $V$. Then, there exists $c \in D\setminus P$ such that $cF$ is properly integral over $\Int_K(M_n(D))$. In particular, if $P$ is principal, then we can take $c=1$.
\end{Thm}
\begin{proof}
By Proposition \ref{Localization prop}, $\Int_K(M_n(D))_P = \Int_K(M_n(V))$, and a similar argument shows that $\Int_K(\Lambda_n(D))_P = \Int_K(\Lambda_n(V))$. So, we can apply Lemma \ref{Global lemma} with $R = \Int_K(M_n(D))$, $S = \Int_K(\Lambda_n(D))$, and $f = F$. Furthermore, if $P$ is principal, then we can assume the denominator $\pi^q$ of $F$ is in $D$ and that $P = \pi D$. In this case, it is immediately seen that $F$ itself is already an element of $S\setminus R$, since $F\in S_Q$ for every prime ideal $Q$ of $D$ different from $P$ and $S=\bigcap_{Q}S_Q$, the intersection ranging over the set of all non-zero prime ideals of $D$.
\end{proof}

\begin{Cor}\label{Int_K not integrally closed}
Let $D$ be a Dedekind domain with finite residue fields. Let $K$ be the field of fractions of $D$ and let $n \geq 2$. Then, $\Int_K(M_n(D))$ is not integrally closed.
\end{Cor}

Finally, we show that the polynomial in Construction \ref{The construction} can be applied to $D$-algebras other than matrix rings. Consider a torsion-free $D$-algebra $A$ that is finitely generated as a $D$-module and such that $A \cap K = D$. If $A$ has a generating set consisting of at most $n$ elements, then each element of $A$ satisfies a monic polynomial in $D[x]$ of degree at most $n$ (see for example \cite[Thm. 1, Chap. V]{Bourbaki} or \cite[Prop. 2.4, Chap. 2]{AtiMac}). It is then a consequence of \cite[Lem.\ 3.4]{Fri2} that $\Int_K(M_n(D)) \subseteq \Int_K(A)$, and thus that the integral closure of $\Int_K(A)$ contains $\Int_K(\Lambda_n(D))$.

\begin{Cor}\label{Algebra case}
Let $D$ and $A$ be as above. Assume that there exists a nonzero prime $P$ of $D$ such that $A/P^q A \cong M_n(D/P^q)$, where $q = |D/P|$. Let $cF$ be as in Theorem \ref{Dedekind case}. Then, $cF$ is properly integral over $\Int_K(A)$. Thus, $\Int_K(A)$ is not integrally closed.
\end{Cor}
\begin{proof}
The polynomial $cF$ is in the integral closure of $\Int_K(A)$ because this integral closure contains $\Int_K(\Lambda_n(D))$. To show that $cF \notin \Int_K(A)$, we will work with localizations. Localize the algebra $A$ in the natural way to produce the $D_P$-algebra $A_P$. Let $\pi$ be the generator of $PD_P$. Then, for all $k > 0$, we have $A_P/\pi^k A_P = A_P/P^k A_P \cong A/P^k A$. In particular, $A_P/\pi^q A_P \cong M_n(D/P^q)$. 

Now, by Proposition \ref{Localization prop}, we see that $\Int_K(A)_Q = \Int_K(A_Q)$ for all nonzero primes $Q$ of $D$, and $\Int_K(A) = \bigcap_Q \Int_K(A)_Q$. So, to prove that $cF \notin \Int_K(A)$, it suffices to show that $cF \notin \Int_K(A_P)$, and since $c$ is a unit of $D_P$, it will be enough to demonstrate that $F \notin \Int_K(A_P)$. Suppose by way of contradiction that $F \in \Int_K(A_P)$. Then, the numerator $G$ of $F$ is such that $G(A_P) \subseteq \pi^q A_P$; equivalently, $G(A_P/\pi^q A_P)$ is 0 mod $\pi^q$. However, because $A_P/\pi^q A_P \cong M_n(D/P^q)$, the argument used Lemma \ref{Big thm first lemma} shows that this is impossible. Thus, we conclude that $F \notin \Int_K(A_P)$.
\end{proof}

\begin{Ex}
Corollary \ref{Algebra case} can be applied when $D = \Z$ and $A$ is a certain quaternion algebra. Let $\bfi$, $\bfj$, and $\bfk$ be such that $\bfi^2 = \bfj^2 = -1$ and $\bfi \bfj = \bfk = -\bfj \bfi$. Let $A$ be either the Lipschitz quaternions
\begin{equation*}
A = \{a_0 + a_1 \bfi + a_2 \bfj + a_3 \bfk \mid a_i \in \Z\}
\end{equation*}
or the Hurwitz quaternions
\begin{equation*}
A = \{a_0 + a_1 \bfi + a_2 \bfj + a_3 \bfk \mid a_i \in \Z \text{ for all } i \text{ or } a_i \in \Z + \tfrac{1}{2} \text{ for all } i\}.
\end{equation*}
In either case, it is a standard exercise (cf.\ \cite[Exer.\ 3A]{GooWar}) that for each odd prime $p$ and each $k > 0$, we have $A/p^k A \cong M_2(\Z/p^k \Z)$. Hence, Corollary \ref{Algebra case} applies, and the polynomial $F$ from Construction \ref{The construction} is properly integral over $\Int_\Q(A)$.

In particular, consider the polynomial $F$ obtained when $p = 3$. The numerator of $F$ is
\begin{align*}
G(x) &= x(x^2+3)(x^2+6)(x-1)((x-1)^2+3)((x-1)^2+6)\\
& \quad \times (x-2)((x-2)^2+3)((x-2)^2+6)(x^2+1)^3(x^2+x+2)^3(x^2+2x+2)^3
\end{align*}
and $F(x) = G(x)/27$, a polynomial of degree 33 that is properly integral over $\Int_{\Q}(A)$.

By contrast, a polynomial $g(x)/27 \in \Int_\Q(A)$ (with $g(x) \in \Z[x]$ not divisible by 3) must have degree at least 36. Indeed, the isomorphisms $A/3^k A \cong M_2(\Z/3^k \Z)$, $k\in\N$, imply that a polynomial $g_1(x)/3^{k_1}$ (with $g_1 \in \Z[x]$ not divisible by 3, and $k_1 > 0$) is in $\Int_\Q(A)$ if and only if it is in $\Int_\Q(M_2(\Z))$. As our results in Section \ref{Null ideals} (such as Corollary \ref{Formula for mu_d}) will show, if $g(x)/27 \in \Int_\Q(M_2(\Z))$, then $\deg(g) \geq 36$. Explicitly, $[(x^9-x)(x^3-x)]^3/27 \in \Int_\Q(M_2(\Z))$ (and hence is in $\Int_\Q(A)$ as well), and there is no polynomial of the form $g(x)/27$ of smaller degree in $\Int_\Q(A)$.
\end{Ex}

\section{\texorpdfstring{Null Ideals and $\pi$-sequences}{Null Ideals and p-sequences}}\label{Null ideals}
Maintain the notation given at the start of Section \ref{Construction section}. Our work so far shows that the polynomial $F$ from Construction \ref{The construction} is properly integral over $\Int_K(M_n(V))$. However, it is possible that there could be a polynomial of degree less than $F$ that is also properly integral over $\Int_K(M_n(V))$. This inspires the next definition.

\begin{Def}
Let $V$ be a DVR with fraction field $K$. A polynomial $f \in K[x]$ that is properly integral over $\Int_K(M_n(V))$ is said to be \textit{optimal} if $f$ is of minimal degree among all properly integral polynomials over $\Int_K(M_n(V))$. 
\end{Def}

We are interested in determining whether our polynomial $F$ is optimal. In general, this is quite hard to do. One way to make progress is to follow the lead of \cite{EvrJoh} and study $P$-sequences for $\Int_K(M_n(V))$ and $\Int_K(\Lambda_n(V))$.

Bhargava introduced $P$-sequences and $P$-orderings for Dedekind domains in \cite{Bha}, and these notions were extended to certain noncommutative rings by Johnson in \cite{Joh}. Among other uses, $P$-sequences and $P$-orderings can be used to give regular bases for rings of integer-valued polynomials (see \cite{Bha}, \cite{Joh}, and \cite{EvrJoh}). For our purposes, $P$ refers to the maximal ideal $\pi V$ of $V$, and we will consider $\pi$-sequences for $\Int_K(M_n(V))$ and $\Int_K(\Lambda_n(V))$.

Recall first Johnson's definition from \cite{Joh}, and its connection to integer-valued polynomials.

\begin{Def}\label{Johnson def}
(\cite[Def. 1.1]{Joh}) Let $K$ be a local field with valuation $v$, $D$ a division algebra over $K$ to which the valuation $v$ extends, $R$ the maximal order in $D$, and $S$ a subset of $R$. Then, a \textit{v-ordering} of $S$ is a sequence $\{a_i \mid i \in \N\} \subseteq S$ with the property that for each $i > 0$ the element $a_i$ minimizes the quantity $v(f_i(a_0, \ldots, a_{i-1})(a))$ over $a \in S$, where $f_0 = 1$ and, for $i > 0$, $f_i(a_0, \ldots, a_{i-1})(x)$ is the minimal polynomial (in the sense of \cite{LamLer}) of the set $\{a_0, a_1, \ldots, a_{i-1}\}$. The sequence of valuations $\{v(f_i(a_0, \ldots, a_{i-1})(a_i)) \mid i \in \N\}$ is called the \textit{v-sequence} of $S$.
\end{Def}

\begin{Prop}\label{Johnon prop}
(\cite[Prop. 1.2]{Joh}) With notation as in Definition \ref{Johnson def}, let $\pi \in R$ be a uniformizing element. Then, the $v$-sequence $\{\alpha_S(i) = v(f_i(a_0, \ldots, a_{i-1})(a_i)) \mid i \in \N \}$ depends only on the set $S$ and not on the choice of $v$-ordering. Moreover, the sequence of polynomials
\begin{equation*}
\{\pi^{-\alpha_S(i)} f_i(a_0, \ldots, a_{i-1})(x) \mid i \in \N\}
\end{equation*}
forms a regular $R$-basis for the $R$-algebra of polynomials integer-valued on $S$.
\end{Prop}

In \cite{EvrJoh}, Evrard and Johnson used these notions to construct $p$-sequences ($p$ a prime of $\Z$) and regular bases for $\Int_\Q(M_2(\Z_{(p)}))$ and its integral closure $\Int_\Q(R_{2, p})$ (here, $R_{2,p}$ is the maximal order of a division algebra of degree $4$ over the field of $p$-adic numbers). We take a slightly different approach and define our $\pi$-sequences with regular bases and optimal polynomials in mind.

\begin{Def}\label{pisequence}
Express polynomials in $K[x]$ in lowest terms, i.e.\ in the form $g(x)/\pi^k$, where $g \in V[x]$, $k \geq 0$, and, if $k > 0$, then $\pi$ does not divide $g$. The $\pi$-sequence $\mu_0, \mu_1, \ldots$ of $\Int_K(M_n(V))$ is the sequence of non-negative integers such that
\begin{equation*}
\mu_d = \max\{k \mid \text{ there exists } g_d(x)/\pi^k \in \Int_K(M_n(V)) \text{ of degree } d\}.
\end{equation*}
In other words, having $\mu_d = k$ means there exists $g_d(x) \in V[x]$ of degree $d$ such that $g_d(x)/\pi^k \in \Int_K(M_n(V))$ with $k$ as large as possible.

The $\pi$-sequence $\lambda_0, \lambda_1, \ldots$ of $\Int_K(\Lambda_n(V))$ is defined similarly.
\end{Def}

\begin{Lem}\label{Regular basis lemma}
For each $d \in \N$, let $f_d$ be a polynomial of degree $d$ in $V[x] \setminus \pi V[x]$, and let $\alpha_d$ be a non-negative integer. If $\{f_d(x)/\pi^{\alpha_d} \mid d \in \N\}$ is a regular $V$-basis for $\Int_K(M_n(V))$ (respectively, $\Int_K(\Lambda_n(V))$), then $\mu_d = \alpha_d$ (respectively, $\lambda_d = \alpha_d$) for all $d$.
\end{Lem}
\begin{proof}
We will prove this for $\Int_K(M_n(V))$ and $\mu_d$; the proof for $\Int_K(\Lambda_n(V))$ and $\lambda_d$ is identical. By means of the notion of characteristic ideals and using \cite[Prop.\ II.1.4]{CahCha}, the sequence $\{g_d(x)/\pi^{\mu_d} \mid d \in \N\}$ forms a regular $V$-basis for $\Int_K(M_n(V))$. Let $\{f_d(x)/\pi^{\alpha_d} \mid d \in \N\}$ be another regular $V$-basis for $\Int_K(M_n(V))$. Then, we must have $\mu_d = \alpha_d$, because the leading coefficients of the elements of two regular bases of the same degree must have the same valuation.
\end{proof}

Relating the previous definition and lemma to the work done in \cite{EvrJoh}, we obtain the following.

\begin{Cor}\label{lambda and p-sequence}
Let $n = 2$, let $p$ be a prime of $\Z$, and let $V = \Z_{(p)}$. Then, $\lambda_d$ is equal to the $p$-sequence for $\Int_\Q(R_{2, p})$ given in \cite[Cor. 2.17]{EvrJoh}.
\end{Cor}

Returning now to the question of optimal properly integral polynomials, we can phrase things in terms of $\pi$-sequences. Since $\Int_K(M_n(V)) \subsetneqq \Int_K(\Lambda_n(V))$, we have $\mu_d \leq \lambda_d$ for all $d$, and there exists $d$ such that $\mu_d < \lambda_d$. Assume we have found the smallest $d$ such that $\mu_d < \lambda_d$. Then, there exists a properly integral polynomial $f(x) = g(x)/\pi^{\lambda_d}$ of degree $d$, and $f(x)$ is optimal.

By Corollary \ref{lambda and p-sequence}, when $n = 2$ and $p$ is a prime of $\Z$, the terms of $\lambda_d$ can be computed by using recursive formulas given in \cite{EvrJoh}. For the general case where $n \geq 2$ and $V$ is a DVR, we now proceed to use the null ideals of the matrix rings $M_n(V/\pi^k V)$ to compute the initial terms of $\mu_d$, although we will not be able to give a formula for the complete sequence. Nevertheless, we will be able to prove (Corollary \ref{When F is optimal}) that the properly integral polynomial $F$ constructed for $\Int_\Q(M_2(\Z_{(p)}))$ is optimal.

We first recall the definition of a null ideal.

\begin{Def}\label{Null ideal def}
Let $R$ be a commutative ring, and let $S$ be a subset of some ring containing $R$. We define the \textit{null ideal} of $S$ in $R$ to be $N_R(S) = \{ f \in R[x] \mid f(S) = 0\}$.
\end{Def}

There is a strong connection between null ideals and integer-valued polynomials, as described in the next lemma. This relationship has been used before in various forms (see \cite{Fri2}, \cite{PerDivDiff}, \cite{Wer3}, and \cite{Wer2}, for example).

\begin{Lem}\label{nullideal and intvalpolynomials}
In the above notation, let $k\in\N$ and $f(x)=g(x)/\pi^k \in K[x]$, for some $g\in V[x]$. Then $f(x)$ is in $\Int_K(M_n(V))$ if and only if $g(x)$ mod $\pi^k$ is in $N_{V/\pi^k V}(M_n(V/\pi^k V))$.
\end{Lem}
\begin{proof}
The polynomial $f(x)$ is integer-valued over $M_n(V)$ if and only if $g(x)$ maps every matrix in $M_n(V)$ to the ideal $\pi^k M_n(V)=M_n(\pi^k V)$. Considering everything modulo $\pi^kV$, we get the stated result, using the fact that $M_n(\pi^kV)\cap V=\pi^k V$.
\end{proof}

Hence, null ideals can give us information about rings of integer-valued polynomials. We are interested in describing generators for the null ideal of $M_n(V/\pi^k V)$ in $V/\pi^k V$. The following polynomials will be crucial in our treatment.

\begin{Not}\label{Phi notation}
For each $n \geq 1$ and each prime power $q$, we define 
\begin{equation*}
\Phi_{q, n}(x) = (x^{q^n} - x)(x^{q^{n-1}} - x) \cdots (x^q - x).
\end{equation*}
With a slight abuse of notation, we will use $\Phi_{q, n}(x)$ to denote the same polynomial over any of the residue rings $V/\pi^k V$, $k\in\N$. The coefficient ring of the polynomial will be clear from the context.  
\end{Not}

Our goal for most of the rest of this section is to prove the next theorem.

\begin{Thm}\label{Phi thm}
Let $n \geq 1$ and let $1 \leq k \leq q$. Then, 
\begin{equation*}
N_{V/\pi^k V}(M_n(V/\pi^k V)) = (\Phi_{q, n}(x), \pi)^k = (\Phi_{q, n}(x)^k, \pi \Phi_{q, n}(x)^{k-1}, \ldots, \pi^{k-1} \Phi_{q, n}(x)). 
\end{equation*}
\end{Thm}

Using different terminology, this theorem was proven for $k=1$ in \cite[Thm. 3]{BraCarLev}; we will revisit that result below in Theorem \ref{BCL Theorem}. When $n = 1$, we have $\Phi_{q, 1}(x) = x^q-x$, and Theorem \ref{Phi thm} is the assertion that $N_{V/\pi^k V}(V/\pi^kV) = (x^q-x, \pi)^k$ for $1 \leq k \leq q$. If, in addition, $V$ is a localization of $\Z$, then this is actually a classical result which can be found in the book of Dickson \cite[Thm.\ 27, p.\ 22]{Dickson}. An alternate modern treatment, which examines the pullback to $\Z[x]$ of the null ideal $N_{\Z/p^k\Z}(\Z/p^k\Z)$, is given in \cite[Thm.\ 3.1]{PerPrimDec}.

The proof of Theorem \ref{Phi thm} is complicated, and involves several stages and preliminary results. We will need to work with different sets of polynomial, common multiples, and least common multiples across the different residue rings $V/\pi^k V$. To help simplify the necessary notation, we adopt the following conventions (the need for all this notation will become apparent as we work through the proof).

\begin{DefNot}\label{phi def}\mbox{}
\begin{itemize}
\item For each $k \geq 1$ let $V_k = V / \pi^k V$ and $N_k = N_{V_k}(M_n(V_k))$. Note that $V_1 = \F_q$.
\item Since $n$ and $q$ will be fixed, let $\Phi = \Phi_{q, n}$.
\item For each $k \geq 1$, let $\phi_k(x)$ be a monic polynomial of minimal degree in $N_k$.
\item For each $k \geq 1$ and each $d \geq 1$, let $\mcP_d(V_k)$ denote the set of monic polynomials of degree $d$ in $V_k[x]$.
\item Let $\mcP_{\leq n}^{{\rm irr}}(\F_q)=\mcP_{\leq n}^{{\rm irr}}$ denote the set of monic irreducible polynomials in $\F_q[x]$ of degree at most $n$.
\item For each $k \geq 1$, $f \in V_k[x]$, and $\iota \in \mcP_{\leq n}^{{\rm irr}}$, we say that $f$ is $\iota$-primary if $f$ is monic and the residue of $f$ in $\F_q[x]$ is a positive power of $\iota$.
\item For each $k \geq 1$, each $d \geq 1$, and each $\iota \in \mcP_{\leq n}^{{\rm irr}}$, let $\mcP^\iota_d(V_k)$ denote the set of $\iota$-primary polynomials in $V_k[x]$ of degree $d$.
\item For each $k \geq 1$, each $d \geq 1$, and each $\iota \in \mcP_{\leq n}^{{\rm irr}}$, let $L_d^\iota(V_k)$ be a monic least common multiple (lcm) for the polynomials in $\mcP^\iota_d(V_k)$. That is, $L_d^\iota(V_k)$ is a monic polynomial in $V_k[x]$ of least degree such that each $f \in \mcP^\iota_d(V_k)$ divides $L_d^\iota(V_k)$. An lcm need not be unique but the degree of an lcm is uniquely determined (see the discussion in \cite{Wer}).
\end{itemize}
\end{DefNot}

In \cite{Fri2}, Frisch described some general properties of null ideals and matrices that we will find very useful.

\begin{Lem}\label{FrischLemma}
\mbox{}
\begin{enumerate}[(1)]
 \item \cite[Lem. 3.3]{Fri2} Let $R$ be a commutative ring, $f\in R[x]$ a monic polynomial and $C\in M_n(R)$ the companion matrix of $f$. Then $N_R(C)=f(x)R[x]$.
\item \cite[Lem. 3.4]{Fri2} Let $D$ be a domain and $f(x) = g(x)/c$, $g\in D[x]$, $c\in D \setminus\{0\}$. Then
$f\in\Int_K(M_n(D))$ if and only if $g$ is divisible modulo $cD[x]$ by all monic polynomials in $D[x]$ of degree $n$.
\end{enumerate}
\end{Lem}

Specializing to our situation, we easily obtain the following corollary.

\begin{Cor}\label{Frisch corollary}
Let $k \geq 1$.
\begin{enumerate}[(1)]
\item Let $f \in V_k[x]$, let $m \in \mcP_n(V_k)$, and let $C \in M_n(V_k)$ be the companion matrix for $m$. Then, $m$ divides $f$ if and only if $f(C) = 0$. 
\item Let $f \in V_k[x]$. Then, $f \in N_k$ if and only if $f$ is divisible by every polynomial in $\mcP_n(V_k)$.
\item The polynomial $\phi_k$ is an lcm for $\mcP_n(V_k)$.
\end{enumerate}
\end{Cor}
\begin{proof}
Part (1) is a restatement of Lemma \ref{FrischLemma} (1). Part (2) follows from Lemmas \ref{nullideal and intvalpolynomials} and \ref{FrischLemma} (2). Finally, for (3), $\phi_k$ is monic by assumption, and is a common multiple for $\mcP_n(V_k)$ because $\phi \in N_k$. But, the minimality of $\deg \phi_k$ means that $\phi_k$ is in fact an lcm for $\mcP_n(V_k)$.
\end{proof}

Thus, we have established a connection between null ideals and least common multiples of the sets $\mcP_n(V_k)$. If we focus on the case $k = 1$, then everything is taking place over the field $\F_q$. In this situation, the aforementioned theorem \cite[Thm.\ 3]{BraCarLev} brings us back to the polynomial $\Phi = \Phi_{q,n}$. 

\begin{Thm}\label{BCL Theorem}\emph{(\cite[Thm.\ 3 \& eq.\ (3.3)]{BraCarLev})}
Let $n \geq 1$ and let $q$ be a prime power. Let $\Phi_{q,n}$ be as in Notation \ref{Phi notation}. 
\begin{enumerate}[(1)]
\item $N_{\F_q}(M_n(\F_q))$ is generated by $\Phi_{q, n}$. 
\item $\Phi_{q, n}$ is the (unique) lcm for $\mcP_n(\F_q)$.
\item The factorization of $\Phi_{q,n}$ into irreducible polynomials is
\begin{equation*}
\Phi_{q,n} = \prod_{\iota \in \mcP_{\leq n}^{{\rm irr}}} \iota^{\lfloor n/\deg \iota \rfloor}.
\end{equation*}
\end{enumerate}
\end{Thm}

Finally, we have all the necessary tools and can proceed with the proof of Theorem \ref{Phi thm}. We break the proof up into three Claims. The first claim shows that it suffices to compare the degrees of the polynomials $\phi_k$ and $\Phi^k$.\\

\noindent\textbf{Claim 1}: To prove Theorem \ref{Phi thm}, it suffices to show that $\deg(\phi_k) \geq \deg(\Phi^k)$ for all $1 \leq k \leq q$.
\begin{proof}
Fix $k$ between $1$ and $q$. Since over $\F_q$ we have $N_1 = (\Phi)$ by Theorem \ref{BCL Theorem} (1), over $V_k$ we have $\Phi^k \in N_k$, and by \cite[Thm. 5.4]{Wer2}, the ideal $N_k$ is equal to $(\phi_k$, $\pi \phi_{k-1}$, $\pi^2 \phi_{k-2}$, $\ldots$, $\pi^{k-1} \phi_1)$. So, to prove Theorem \ref{Phi thm}, it will be enough to show that we can take $\phi_k = \Phi^k$, and doing so is valid if $\deg(\phi_k) = \deg(\Phi^k)$.

Now, by Corollary \ref{Frisch corollary} part (3), $\phi_k$ is an lcm for $\mcP_n(V_k)$. We will show that $\Phi^k$ is a common multiple for $\mcP_n(V_k)$, i.e.\ that each $f\in \mcP_n(V_k)$ divides $\Phi^k$. To do this, let $f \in \mcP_n(V_k)$ and let $C \in M_n(V_k)$ be the companion matrix for $f$. 

Recall that we have a canonical projection map from $M_n(V_k)$ to $M_n(\F_q)$, whose kernel is $\pi M_n(V_k)=M_n(\pi V_k)$. Over the residue field $\F_q$, by Theorem \ref{BCL Theorem} the polynomial $\Phi$ is zero on the matrix obtained by reducing the entries of $C$ modulo $\pi$. It follows that over $V_k$ we have $\Phi(C)\in M_n(\pi V_k)$. Hence, $\Phi(C)^k = 0$ in $M_n(V_k)$. 

By Corollary \ref{Frisch corollary} (1), $f$ divides $\Phi^k$, and since $f$ was arbitrary, we conclude that $\Phi^k$ is a common multiple for $\mcP_n(V_k)$. Since $\phi_k$ is an lcm for $\mcP_n(V_k)$, we have $\deg(\phi_k) \leq \deg(\Phi^k)$.

Thus, to complete the proof, it suffices to show that $\deg(\phi_k) \geq \deg(\Phi^k)$.
\end{proof}

Next, we argue that it is enough just to focus our attention on $\iota$-primary polynomials.\\

\noindent\textbf{Claim 2}: To prove Theorem \ref{Phi thm}, it suffices to show that for all $1 \leq k \leq q$ and all $\iota \in \mcP_{\leq n}^{{\rm irr}}$, we have $\deg(L_D^\iota(V_k)) \geq kD$, where $D = \deg(\iota) \lfloor \frac{n}{\deg(\iota)}\rfloor$.
\begin{proof}
Let $D=\deg(\iota)\lfloor \frac{n}{\deg(\iota)}\rfloor$. By Theorem \ref{BCL Theorem} (3), we have 
\begin{equation}\label{1st phi eq}
\Phi^k = \prod_{\iota \in \mcP_{\leq n}^{{\rm irr}}} \iota^{k\lfloor n/\deg \iota \rfloor}.
\end{equation} 
Moreover, by \cite[Thm.\ 5.1]{Wer} we know that the polynomial $\prod_{\iota \in \mcP_{\leq n}^{{\rm irr}}} L_D^\iota(V_k)$ is an lcm for $\mcP_n(V_k)$. Thus, we can take
\begin{equation}\label{2nd phi eq}
\phi_k = \prod_{\iota \in \mcP_{\leq n}^{{\rm irr}}} L_D^\iota(V_k).
\end{equation}
Comparing \eqref{1st phi eq} and \eqref{2nd phi eq} gives us a method of attack: we can prove that $\deg(\phi_k) \geq \deg(\Phi^k)$ by showing that for each $\iota$, we have 
\begin{equation}\label{3rd phi eq}
\deg(L_D^\iota(V_k)) \geq \deg(\iota^{k\lfloor n/\deg \iota \rfloor}) = k\deg(\iota)\lfloor \tfrac{n}{\deg(\iota)}\rfloor=kD.  \qedhere
\end{equation}
\end{proof}

To complete the proof of Theorem \ref{Phi thm}, all that remains is to justify the inequality (\ref{3rd phi eq}) from the previous claim.\\

\noindent\textbf{Claim 3}: Let $1 \leq k \leq q$ and let $\iota \in \mcP_{\leq n}^{{\rm irr}}$. Let $D=\deg(\iota)\lfloor \frac{n}{\deg(\iota)}\rfloor$.  Then, $\deg(L_D^\iota(V_k)) \geq kD$.
\begin{proof}
For the final stage of the proof, $k$ and $\iota$ are fixed, so we can simplify the notation. Let $d = \deg \iota$, let $D = d \lfloor \tfrac{n}{d} \rfloor$, let $\mcP = \mcP^\iota_D(V_k)$, and let $f = L_D^\iota(V_k)$. We need to prove that $\deg f \geq kD$. Unless stated otherwise, calculations take place mod $\pi^k$.

Choose a $k$-element subset $\{a_1, \ldots, a_k\}$ from $\F_q$. For each $1 \leq j \leq k$, let $m_j(x) = (\iota(x))^{\lfloor n/d \rfloor} - \pi a_j \in \mcP$, and let $C_j \in M_D(V_k)$ be the $D \times D$ companion matrix for $m_j$. Then, for all $1 \leq j \leq k$, we have $m_j(C_j) = 0$ and $(\iota(C_j))^{\lfloor n/d \rfloor} = \pi a_j I$. This latter relation implies that $m_j(C_{j'}) = \pi(a_{j'} - a_j)I$ for all $1 \leq j, j' \leq k$.

Now, $m_1$ divides $f$ because $f$ is an lcm for $\mcP$. Since both $m_1$ and $f$ are monic, there exists a monic $f_1 \in V_k[x]$ such that $f = m_1 f_1$. If $k = 1$, then we are done, so assume that $k \geq 2$. In that case, $m_2$ also divides $f$, so
\begin{equation*}
0 = f(C_2) = m_1(C_2) f_1(C_2) = \pi (a_2 - a_1) f_1(C_2).
\end{equation*}
This equality occurs mod $\pi^k$, and $a_2 - a_1$ is a unit mod $\pi$ (hence is a unit mod $\pi^k$), so $f_1(C_2) \equiv 0$ mod $\pi^{k-1}$. Thus, $m_2$ divides $f_1$ mod $\pi^{k-1}$, so in $V_k[x]$ we may write $f_1 = m_2 f_2 + \pi^{k-1} g_1$, where $f_2, g_1 \in V_k[x]$, $f_2$ is monic, and $\deg g_1 < \deg f_1 = \deg f - D$.

At this point, we have
\begin{equation*}
f = m_1 f_1 = m_1 (m_2 f_2 + \pi^{k-1} g_1) = m_1 m_2 f_2 + \pi^{k-1} m_1 g_1.
\end{equation*}
If $k = 2$, we are done; if not, applying the same argument as above yields
\begin{align*}
0 &= f(C_3)\\
&= m_1(C_3) m_2(C_3) f_2(C_3) + \pi^{k-1} m_1(C_3) g_1(C_3)\\
&= \pi^2(a_3-a_1) (a_3-a_2) f_2(C_3) + 0.
\end{align*}
Since $(a_3-a_1) (a_3-a_2)$ is a unit mod $\pi^k$, we have $f_2(C_3) \equiv 0$ mod $\pi^{k-2}$. The same steps as before will give us
\begin{equation*}
f = m_1 m_2 m_3 f_3 + \pi^{k-1} m_1 g_1 + \pi^{k-2} m_1 m_2 g_2
\end{equation*}
where $f_3, g_2 \in V_k[x]$, $f_3$ is monic, and $\deg g_2 < \deg f_2 = \deg f - 2D$. 

Since $k \leq q$, the product $(a_k - a_1)(a_k - a_2) \cdots (a_k - a_{k-1})$ will always be a unit mod $\pi^k$. Thus, we can continue this process as long as necessary, ultimately resulting in the expansion
\begin{equation*}
f = m_1 m_2 \cdots m_k f_k + \pi g
\end{equation*}
where $f_k, g \in V_k[x]$, $f_k$ is monic, and $\deg g < \deg f$. Since each $m_j$ has degree $D$, we conclude that $\deg f \geq kD$, as required.
\end{proof}

\begin{Rem}
Theorem \ref{Phi thm} does not hold once $k > q$; we demonstrate this by example below. Examining the proof gives some indication why. The final stage of the proof relied on the fact that the product $(a_k - a_1)(a_k - a_2) \cdots (a_k - a_{k-1})$ is nonzero mod $\pi$, and this will not be true once $k > q$. Products of this form arise naturally with $P$-orderings \cite{Bha}, and illustrate once again the close connections between $P$-orderings and integer-valued polynomials.
\end{Rem}

\begin{Ex}
Theorem \ref{Phi thm} is false for $k=q+1$. Let
\begin{align*}
\theta(x) &= \Phi/\big(\prod_{a \in \F_q}(x-a)^n\big) = \prod_{\iota \in \mcP_{\leq n}^{{\rm irr}}, \deg \iota \geq 2} \!\! \iota(x)^{\lfloor n/ \deg \iota \rfloor},\\%
\ell(x) &= x^{n-1}\prod_{a \in \F_q} (x^n + \pi a),\\
L(x) &= \prod_{a \in \F_q} \ell(x-a), \text{ and }\\
\psi(x) &= L(x) \theta(x)^{q+1}
\end{align*}
(cf.\ Construction \ref{The construction}). We claim that $\psi \in N_{q+1}$. Let $C$ be the companion matrix for a polynomial $m \in \mcP_n(V_{q+1})$. If $m \not\equiv (x-a)^n $ mod $\pi$ for all $a \in \F_q$, then $\theta(C)\equiv0$ mod $\pi$, so $\psi(C) \equiv 0$ mod $\pi^{q+1}$. So, assume $m \equiv (x-a)^n$ mod $\pi$ for some $a \in \F_q$.

Assume first that $a = 0$, and consider $m$ mod $\pi^2$. There exists $b \in \F_q$ such that the constant term of $m$ is equivalent to $-\pi b$ mod $\pi^2$. Consequently, $C^n + \pi b I$ is divisible by $\pi C$ mod $\pi^2$. It follows that $C^{n-1}(C^n + \pi b I) \equiv 0$ mod $\pi^2$, and so $\ell(C) \equiv 0$ mod $\pi^{q+1}$. By translation, $L(C) \equiv 0$ mod $\pi^{q+1}$ regardless of the choice of $a$. We conclude that $\psi(C) \equiv 0$ mod $\pi^{q+1}$ for all companion matrices $C$. Thus, $\psi \in N_{q+1}$.

However, one may compute that $\deg \psi = (q+1)\deg \Phi - q$. Since $\deg \phi_{q+1} \leq \deg \psi < \deg (\Phi^{q+1})$, Theorem \ref{Phi thm} does not hold for $k = q+1$.
\end{Ex}

We close the paper by once again considering $\pi$-sequences and optimal polynomials (see the definitions given at the start of this section). By using Theorem \ref{Phi thm} , we can give a succinct formula for the initial terms of the $\pi$-sequence $\mu_d$.

\begin{Cor}\label{Formula for mu_d}
The $\pi$-sequence $\mu_d$ for $\Int_K(M_n(V))$ satisfies $\mu_d = \lfloor d/\deg \Phi_{q, n} \rfloor$ for $0 \leq d \leq q\cdot \deg \Phi_{q, n}$.
\end{Cor}
\begin{proof}
The polynomial $g(x)/\pi^k \in K[x]$ (where $g \in V[x]$ and $\pi$ does not divide $g$) is in $\Int_K(M_n(V))$ if and only if $g(x)$ mod $\pi^k$ is in the null ideal $N_k$ (Lemma \ref{nullideal and intvalpolynomials}). Moreover, $\mu_d$ is equal to the maximum $k$ such that there exists $g(x)/\pi^k \in \Int_K(M_n(V))$ of degree $d$. It follows that, for any $k > 0$, we have $\mu_{\deg \phi_k} = k$, and $\mu_d < k$ for $d < \deg \phi_k$. By Theorem \ref{Phi thm}, $\deg \phi_k = k \deg \Phi_{q,n}$ for $1 \leq k \leq q$. Hence, the sequence $\mu_d$ begins
\begin{equation*}
\underbrace{0, \ldots, 0}_{\deg \Phi_{q, n} \text{ terms}}, \; \underbrace{1, \ldots, 1}_{\deg \Phi_{q, n} \text{ terms}}, \; \ldots, \; \underbrace{q-1, \ldots, q-1}_{\deg \Phi_{q, n} \text{ terms}}, \; q,
\end{equation*}
which matches the stated formula.
\end{proof}

In general, it is harder to describe the $\pi$-sequence $\lambda_d$ of $\Int_K(\Lambda_n(V))$. Thankfully, formulas for the case $n = 2$ and $V = \Z_{(p)}$ are given in \cite{EvrJoh}, and we can use these to show that our polynomial $F$ is optimal in that case.

\begin{Cor}\label{When F is optimal}
Let $p$ be a prime of $\Z$. Then, the polynomial $F$ given by Construction \ref{The construction} for $\Int_\Q(M_2(\Z_{(p)}))$ is optimal.
\end{Cor}
\begin{proof}
The degree of $\Phi_{p, 2}$ is $p^2+p$, so Corollary \ref{Formula for mu_d} tells us that the $p$-sequence $\mu_d$ of $\Int_\Q(M_2(\Z_{(p)}))$ satisfies $\mu_d = \lfloor d/(p^2+p) \rfloor$ for $0 \leq d \leq p^3+p^2$. The $p$-sequence $\lambda_d$ of $\Int_\Q(\Lambda_2(\Z_{(p)}))$ can be computed via recursive formulas given in \cite[Prop.\ 2.13, Prop.\ 2.10, Cor.\ 2.17]{EvrJoh}. An elementary, but tedious, calculation (which we omit for the sake of space) shows that for $0 \leq d < p^3 + p^2 -p$, we have $\lambda_d = \lfloor d/(p^2 + p) \rfloor < p$, and $\lambda_{p^3+p^2-p} = p$. Thus, the smallest $d$ for which $\mu_d < \lambda_d$ is $d = p^3+p^2-p$. A routine computation shows that $\deg F = p^3 + p^2 - p$, so we conclude that $F$ is optimal.
\end{proof}

It is an open problem to determine whether Corollary \ref{When F is optimal} holds in the general case.

\begin{Ques}
Let $V$ be a DVR with fraction field $K$ and residue field $\F_q$, and let $n \geq 2$. Is the polynomial $F$ given by Construction \ref{The construction} optimal? To prove this, it would suffice to show that $\lambda_d=\mu_d$ for all $d < \deg F = q \deg \Phi_{q, n} - q$. We will not include the proof, but we have been able to determine that $\lambda_d = \mu_d = 0$ for $0 \leq d < \deg \Phi_{q, n}$. However, we have not been able to prove that equality holds for larger $d$ (although we suspect that this is the case).
\end{Ques}

\section*{Acknowledgements}
\noindent This research has been supported by  the grant ``Assegni Senior'' of the University of Padova. The authors wish to thank the referee for several suggestions which improved the paper.
\vskip0.8cm

\end{document}